\documentclass[12pt,a4paper,reqno]{amsart}
\usepackage{amsmath,amsfonts,amssymb,amsthm,amscd}
\usepackage{pst-node,pstricks,pst-tree}

\input{comment.sty}
\includecomment{MM}
\includecomment{CL}

\renewcommand{\a}{\alpha}
\renewcommand{\b}{\beta}
\newcommand{\g}{\gamma}
\newcommand{\G}{\Gamma}
\renewcommand{\d}{\delta}
\newcommand{\D}{\Delta}

\renewcommand{\l}{\lambda}

\newcommand{\m}{\mu}
\newcommand{\n}{\nu}
\renewcommand{\o}{\omega}
\renewcommand{\O}{\Omega}
\renewcommand{\r}{\rho}
\newcommand{\s}{\sigma}

\renewcommand{\t}{\tau}

\newcommand{\e}{\varepsilon}

\newcommand{\x}{\xi}

\newcommand{\y}{\eta}

\newcommand{\Oc}{{\mathcal O}}
\newcommand{\C}{{\mathbb C}}

\newcommand{\N}{{\mathbb N}}
\newcommand{\R}{{\mathbb R}}

\newcommand{\Z}{{\mathbb Z}}
\renewcommand{\Pr}{{\mathbb P}}
\renewcommand{\H}{{\mathbb H}}
\newcommand{\Hc}{{\mathcal H}}
\newcommand{\Fc}{{\mathcal F}}
\newcommand{\Sc}{{\mathcal S}}

\newcommand{\gr}{{\rm gr}}
\newcommand{\sgr}{{\rm sgr}}

\newtheorem{theorem}{Theorem}[section]
\newtheorem{proposition}[theorem]{Proposition}
\newtheorem{lemma}[theorem]{Lemma}
\newtheorem{corollary}[theorem]{Corollary}

\theoremstyle{definition}
\newtheorem{definition}[theorem]{Definition}

\theoremstyle{remark}
\newtheorem{remark}[theorem]{Remark}

\newtheorem{question}[theorem]{Question}

\begin{document}

\title[Dimension formula]
      {Dimension of harmonic measures \\in hyperbolic spaces}
\author{Ryokichi Tanaka}
\address{Mathematical Institute, Tohoku University, 6-3 Aza-Aoba, Aramaki, Aoba-ku, Sendai 980-8578 JAPAN}
\email{rtanaka@m.tohoku.ac.jp}
\date{\today}

\maketitle

\begin{abstract}
We show exact dimensionality of harmonic measures associated with random walks on groups acting on a hyperbolic space under finite first moment condition, and establish the dimension formula by the entropy over the drift.
We also treat the case when a group acts on a non-proper hyperbolic space acylindrically.
Applications of this formula include continuity of the Hausdorff dimension with respect to driving measures and Brownian motions on regular coverings of a finite volume Riemannian manifold.
\end{abstract}

\section{Introduction}

Let $(X, d)$ be a metric space, and $\n$ be a Borel measure on $X$.
The {\it Hausdorff dimension} of the measure $\n$ is the smallest Hausdorff dimension of sets of full measure $\n$.
It is of interest especially when the measure has a dynamical origin.
In this paper, we are concerned with the {\it harmonic measure} associated with a random walk on a group acting on a Gromov hyperbolic space and its Hausdorff dimension.

Let us consider a geodesic hyperbolic space $X$ which we assume {\it proper}, i.e., every closed ball is compact,
--- we deal with a {\it non-proper} space later ---, 
and a countable group $\G$ of isometries of $X$.
Standard examples include the real hyperbolic space $\H^n$ with a discrete subgroup $\G$ of isometries of $\H^n$ and a Cayley graph of a word hyperbolic group with action of the group itself.
For a probability measure $\m$ on the group $\G$, we consider an independent sequence of random group elements $x_1, x_2, \dots$, with common distribution $\m$, and the product 
$$w_n:=x_1 \cdots x_n,$$ 
where $w_0:=id$.
Fix a base point $o$ in the space $X$, then we call $\{w_n o\}_{n=0}^\infty$ a {\it random trajectory} on $X$ starting from $o$.
Associated with the hyperbolic space $X$, one can define the boundary $\partial X$ (the Gromov boundary) as a set of equivalence classes of divergent sequences.
The boundary $\partial X$ admits a canonical quasi-metric $\r$, which is bi-H\"olderian to some metric. 
It is compact when $X$ is proper. 
If the random trajectory $\{w_n o\}_{n=0}^\infty$ converges to a (random) point in the boundary $\partial X$ almost surely, then we call the distribution of limiting points the {\it harmonic measure} (or the {\it stationary measure}), and denote it by $\n_\m$, which depends on $\m$ the step distribution of the random walk.
The boundary $\partial X$ endowed with the harmonic measure $\n_\m$ is of special interest in the context of representation of (bounded) harmonic functions on the group $\G$ --- the {\it Poisson boundary} of $(\G, \m)$.
We discuss a finer property of the space $(\partial X, \n_\m)$. 
The main objective is to establish a formula of Hausdorff dimension in terms of random walk.

\subsection{Main results}
We formulate our main results.
First, let us denote by $\gr(\m)$ the group generated by the support of $\m$, and by $\sgr(\m)$ the semigroup generated by the support of $\m$.
They coincide for example when the measure $\m$ is {\it symmetric}, i.e., $\m(g)=\m(g^{-1})$, but they do not coincide in general.
We call a group $\G$ {\it non-elementary} if its orbit of a (or, equivalently every) point has infinitely many limit points in the boundary $\partial X$.
If the group $\gr(\m)$ is non-elementary, then the harmonic measure $\n_\m$ is well-defined, i.e., the random trajectory $\{w_n o\}_{n=0}^\infty$ converges to a point in the boundary $\partial X$ almost surely; moreover, $\n_\m$ is non-atomic and is the unique {\it $\m$-stationary measure}, i.e.\ a unique probability measure to satisfy
$$
\n_\m=\sum_{g \in \G}\m(g)g\n_\m,
$$
where $g\n_\m$ denotes the pushforward of $\n_\m$ by the group action $g$ on the boundary \cite{K00}.
We often impose a stronger assumption that the semigroup $\sgr(\m)$ is a non-elementary group.
In what follows we make a clear distinction between $\gr(\m)$ and $\sgr(\m)$.

Next, for a probability measure $\m$ on $\G$, 
we say that a probability measure $\m$ has {\it finite first moment} when
$$
L(\m):=\int_\G d(o, g o)d\m(g) < \infty.
$$
Note that the condition does not depend on the choice of base point $o$ by the triangular inequality.
Let $H(\m):=-\sum_{g \in \G}\m(g)\log \m(g)$ (the Shannon entropy of $\m$).
We define the {\it entropy} $h_\m$ and the {\it drift} $l_\m$ by
$$
h_\m:=\lim_{n \to \infty}\frac{1}{n}H(\m^{\ast n}), \ \ \ l_\m:=\lim_{n \to \infty}\frac{1}{n}L(\m^{\ast n}),
$$
where $\m^{\ast n}$ is the $n$-th convolution of $\m$,
and each limit exists by the subadditivity and is finite when $H(\m)<\infty$, and when $L(\m)<\infty$, respectively.
It is known that if the group $\gr(\m)$ is non-elementary, then the entropy $h_\m$ and the drift $l_\m$ are positive \cite{K00}.
We say that $\G$ has {\it exponential growth relative to} $X$ 
if there exists a constant $C$ such that for all $r \ge 0$, we have
\begin{equation}\label{exponential-growth}
\sharp\{g \in \G \ : \ d(o, g o) \le r\} \le Ce^{Cr}.
\end{equation}
Again the condition does not depend on the choice of base point $o$.
This holds for instance the case when the group $\G$ is a discrete subgroup of isometries of the real hyperbolic space $\H^n$ and the case of a Cayley graph of finitely generated group.

\begin{theorem}\label{exact}
Let $(X, d)$ be a hyperbolic proper geodesic metric space, and $\G$ a countable group of isometries of $X$ having exponential growth relative to $X$.
If the semigroup $\sgr(\m)$ generated by the support of $\m$ is a non-elementary subgroup and $\m$ has finite first moment,
then the harmonic measure $\n_\m$ on $\partial X$ is exact dimensional, i.e.,
$$
\lim_{r \to 0}\frac{\log \n_\m\left(B(\x, r)\right)}{\log r}=\frac{h_\m}{l_\m}
$$
for $\n_\m$-almost every $\x$ in $\partial X$.
In particular, we have
$$
\dim \n_\m=\frac{h_\m}{l_\m},
$$
and consequently, $\dim \n_\m>0$.
Here we denote by $B(\x, r)$ the ball of radius $r$ centred at $\x$ with the quasi-metric in the boundary $\partial X$, and by $\dim \n_\m$ the Hausdorff dimension of $\n_\m$.
\end{theorem}

When we have a bounded geometry-type assumption on $X$, we can strengthen the result to semigroup random walks.
A metric space $(X, d)$ has {\it bounded growth at some scale} when there exist constants $r, R$ $(0 < r < R)$ and $N$ such that every open ball of radius $R$ is covered by $N$ open balls of radius $r$ \cite{BS}.
For example, a class of hyperbolic spaces of bounded growth at some scale includes every complete simply connected Riemannian manifold with uniformly bounded negative sectional curvature, and every Cayley graph of word hyperbolic group.

\begin{theorem}\label{exact-semigroup}
Let $(X, d)$ be a hyperbolic proper geodesic metric space of bounded growth at some scale, and $\G$ be of exponential growth relative to $X$.
If the group $\gr(\m)$ generated by the support of $\m$ is a non-elementary subgroup and $\m$ has finite first moment,
then the harmonic measure $\n_\m$ on $\partial X$ is exact dimensional and
$$
\dim \n_\m=\frac{h_\m}{l_\m}.
$$
In particular, $\dim \n_\m>0$.
\end{theorem}

In fact, our proof is general enough to extend to groups acting on non-proper hyperbolic spaces.
But in the non-proper setting, many groups of interest do not satisfy the condition of exponential growth (\ref{exponential-growth}).
Let $X$ be a separable geodesic (but not necessarily proper) hyperbolic space.
In this case, note that the corresponding boundary $\partial X$ is a separable complete metric space; yet not necessarily compact.
In \cite{MT}, Maher and Tiozzo showed 
that the semigroup ${\rm sgr}(\m)$ is a non-elementary subgroup of $\G$,
then the harmonic measure is well-defined; namely,
for the random walk $\{w_n\}_{n=0}^\infty$ with the step distribution $\m$, 
almost every trajectory $\{w_n o\}_{n=0}^\infty$ converges to a point in the boundary $\partial X$,
and also the distribution $\n_\m$ of the limiting point is non-atomic and the unique $\m$-stationary measure on $\partial X$.
They proved moreover, if the group $\G$ acts on $X$ {\it acylindrically}, and $\m$ has finite entropy and finite logarithmic moment, then the boundary $\partial X$ endowed with the harmonic measure $\n_\m$ is in fact the Poisson boundary of $(\G, \m)$ \cite[Theorem 1.5]{MT}.
Recall that a group $\G$ acts on $X$ {\it acylindrically} if the group acts on $X$ by isometries and satisfies the following:
for every $K \ge 0$, there exist constants $R, N$ such that for all two points $x$ and $y$ in $X$ with $d(x, y) \ge R$ 
there are at most $N$ group elements $g$ in $\G$ satisfying that $d(x, g x) \le K$ and $d(y, g y) \le K$.
If the group admits a proper action on a hyperbolic space, then the action is acylindrical by definition.
Also notice that if the group acts on a proper hyperbolic space acylindrically, then the action is in fact proper.
For example, the mapping class group of an orientable surface acts on the curve complex acylindrically; see \cite{O} and references therein for more examples and their basic properties.

\begin{theorem}\label{acyl}
Let $\G$ be a countable group of isometries of a separable geodesic hyperbolic space $X$. 
If $\G$ acts on $X$ acylindrically, and the semigroup ${\rm sgr(\m)}$ generated by the support of $\m$ is a non-elementary subgroup of $\G$ and $\m$ has finite entropy and finite first moment,
then the harmonic measure $\n_\m$ is exact dimensional and
$$
\dim \n_\m=\frac{h_\m}{l_\m}.
$$
In particular, $\dim \n_\m>0$.
\end{theorem}

\subsection{Historical background}

There is a long history on the question to obtain the Hausdorff dimension of a measure, or more strongly, to show that it is exact dimensional in smooth and non-smooth dynamical systems, where the formula which relates dimension, entropy and Lyapunov exponents (this corresponds to the drift) has been investigated (e.g., \cite{Young}).
Recall that for a probability measure $\n$ on a metric space, we call the $\liminf$ (resp.\ the $\limsup$) of $\log \n\left(B(\x, r)\right)/\log r$ as $r \to 0$ the {\it lower local dimension} (resp.\ the {\it upper local dimension}) at $\x$, and call the value the {\it local dimension} at $\x$ when these two coincide.
We call the measure $\n$ {\it exact dimensional} when the local dimension exists at $\n$-almost every $\x$ and is a constant.
For a probability measure invariant under a $C^2$-diffeomorphism on a compact smooth Riemannian manifold, 
Ledrappier and Young showed that the measure has the local dimensions along stable and unstable local manifolds respectively \cite{LY}.
As a special case, they showed that if the measure is hyperbolic, then its upper local dimension is at most the sum of those two local dimensions.
The problem had been open for a while as to whether the local dimension of it is precisely the sum --- known as the Eckmann-Ruelle conjecture ---; this was later confirmed by Barreira, Pesin and Schmeling for a $C^{1+\a}$-diffeomorphism \cite{BPS}.
See \cite[Chap.\ 8]{P} for more on a background on this problem.
Concerning non-smooth systems, 
Feng and Hu showed the exact dimensionality for self-similar measures in $\R^d$ \cite{FengHu}, and
Hochman made notable progresses about the dimension of self-similar measures on the line, and their applications to the absolutely continuous versus singular question \cite{Hochman}. 
Recently, B\'ar\'any and Ka\"enm\"aki proved that every self-affine measure on the plane is exact dimensional and obtained a formula for the Hausdorff dimension \cite{BK}.
They also proved that every (quasi-) self-affine measure on $\R^d$ is exact dimensional under some condition on the linear parts.
See \cite{Hochman} and \cite{BK} for a background and recent progress in this direction.

Ledrappier introduced this problem in the context of random walks.
He established the formula that a dimension of $\n_\m$ coincides with $h_\m/(2\l_\m)$ for random walks on discrete subgroups of $SL(2, \C)$, where $\l_\m$ is the Lyapunov exponent \cite{L83}.
Note, however, that there the dimension of the harmonic measure $\n_\m$ is defined by the constant to which 
$
\log \n_\m\left(B(\x, r)\right)/\log r
$
converges in measure as $r \to 0$ instead of almost everywhere in $\n_\m$.
For finitely generated free groups, he showed that the harmonic measure $\n_\m$ is exact dimensional and $\dim \n_\m=h_\m/l_\m$ when $\m$ has finite first moment \cite[Theorem 4.15]{L01}.
Kaimanovich strengthened this result to arbitrary free groups (including an infinitely generated case) and established the dimension formula for a general class of processes on trees \cite{K98}.
Le Prince generalized this method to a discrete subgroup of isometries of proper geodesic Gromov hyperbolic space $X$.
He showed that for a probability measure $\m$ of finite first moment,
the upper local dimension is at most $h_\m/l_\m$,
and $\dim \n_\m \le h_\m/l_\m$ \cite{LP07}.
This was applied to show that 
$\dim \n_\m$ can be smaller than any given value for a (symmetric) probability measure $\m$.
He also showed that when it comes to a weaker notion of dimension --- the {\it box dimension} ---, then the box dimension of $\n_\m$ equals $h_\m/l_\m$ \cite[Theorem 3.1]{LP08}.
Soon after that, Blach\`ere, Haissinsky and Mathieu proved that for a word hyperbolic group (with a general class of metrics),
if $\m$ is a symmetric finitely supported probability measure, 
then the harmonic measure $\n_\m$ is exact dimensional and established the formula $\dim \n_\m=h_\m/l_\m$ \cite[Theorem 1.3 and Corollary 1.4]{BHM11}.
Their proof is based on the fact that the Green function is almost multiplicative along a geodesic (Ancona's inequality), which implies that $\n_\m$ is doubling.

A finer dimensional property of $\n_\m$ is investigated for a word hyperbolic group (with a word metric) in \cite{T}.
There it is proved that if $\m$ is finitely supported (not necessarily symmetric) and the support generates the whole group as a semigroup, 
then the harmonic measure $\n_\m$ is exact dimensional, and the singularity spectrum of $\n_\m$ (the multifractal spectrum) is obtained.
A recent notable progress is made by Gou\"ezel, Math\'eus and Maucourant \cite{GMM}.
One of their results implies that if $\G$ is a word hyperbolic group (with a word metric) which is not virtually free, $\m$ has a finite support (more generally, has superexponential moment) and the support generates $\G$ as a semigroup, then $\dim \n_\m < D_\G$, where $D_\G$ denotes the Hausdorff dimension of the boundary $\partial \G$.
Their approach is also based on the estimate of Green functions (Ancona's inequality and its strengthened version), which is used to show that the harmonic measure $\n_\m$ has the maximal dimension only when it is in the unique measure class of maximal dimension (a Gibbs property).
See the introduction of \cite{KLP} for more about the historical background of dimension formula for random walks on groups.

Concerning an extension to groups acting on a non-proper hyperbolic space, 
recently, Das, Simmons and Urba\'nski construct and study Patterson-Sullivan measures for groups acting on a class of non-proper hyperbolic spaces \cite{DSU}.
They investigate exact dimensionality of such measures on the corresponding boundary.
A class of groups acting on a non-proper hyperbolic space acylindrically covers various important examples which are out of proper settings.
Mathieu and Sisto study random walks on those groups and obtain the regularity of the entropy and the drift among others
\cite{Mathieu-Sisto}.

\subsection{Overview of the proof}
Let us illustrate the proof of Theorem \ref{exact} for a probability measure $\m$ of finite first moment. 
The bound for the upper local dimension of $\n_\m$ is known to be $h_\m/l_\m$ in a fairly general setting by Le Prince \cite{LP07}[Proposition 2.3] (see Theorem \ref{upperbound} below); hence all the point is to show the corresponding lower bound.
Since the boundary $\partial X$ is not necessarily a totally disconnected Cantor-like set, Ledrappier's argument for free groups is not applicable in a straightforward way; and also since the $\m$ is only assumed to have finite first moment, unlike in the case proved by Blach\`ere {\it et al.}\ \cite{BHM11},
in this case there is no Ancona's inequality for the Green function.
Here we extensively make use of ergodicity (or its substitute) and stationarity of the harmonic measure $\n_\m$, and the geodesic tracking property of random walks on hyperbolic spaces.

Let $A_\e$ be an event defined in the space of samples $\O$ as on that event, $\m^{\ast n}(w_n)$ and $d(o, w_n o)$ have the right asymptotics, i.e., $-\log \m^{\ast n}(w_n) \sim n h_\m$ and $d(o, w_n o) \sim n l_\m$ within an error $\e n$ for all large enough $n$.
Given a $\n_\m$-almost every limiting point $\x$ in the boundary $\partial \G$ and a (unit speed) geodesic ray, also denoted by $\x$, starting from $o$ towards that point, we show that for every $\e>0$ and all large enough $n$,
\begin{equation*}\label{A}
\Pr \left(A_\e \cap \{d(w_n o, \x(l_\m n)) \le \e n)\}\right) \le e^{-n (h_\m-\e)}.
\end{equation*}
This already suffices for free groups \cite[Theorem 4.15]{L01} (where we can use a martingale technique to conclude the desired estimate), and also for the box dimension \cite[Theorem 3.1]{LP08}.
In order to proceed in general hyperbolic spaces,
we transfer the condition $A_\e$ to the boundary set $\partial X$.
Let $F_\e$ be the set in $\partial X$ where conditioned on almost every point in $F_\e$, the event $A_\e$ occurs with a positive probability.
Here to define $F_\e$, we use a system of conditional probability measures (Rohlin's theory of disintegration).
The construction will imply that the set $F_\e$ has a positive $\n_\m$-measure.
Then we show that
for every $\e > 0$, and for $\n_\m$-almost every $\x$ in $\partial X$,
\begin{equation}\label{intF}
\liminf_{r \to 0}\frac{\log \n_\m\left(F_\e \cap B(\x, r)\right)}{\log r}\ge \frac{h_\m}{l_\m} -\e.
\end{equation}
See Theorem \ref{F} for the precise statement.
This enables us to show the lower bound of local dimension by using ergodicity and stationarity of the harmonic measure $\n_\m$,
where we use the assumption that the semigroup $\sgr(\m)$ is a group; and conclude Theorem \ref{exact}.
The above (\ref{intF}) follows from the construction of such a set $F_\e$ and the geodesic tracking property of random walks.

Theorem \ref{F}, in fact, holds when the group $\gr(\m)$ is non-elementary.
In the case of $\H^n$, we make use of the fact that the boundary is bi-Lipschitz equivalent to the standard sphere $S^{n-1}$; 
there we are able to use instead of ergodicity of $\n_\m$, the Borel density theorem. 
More generally, if the space $X$ has bounded growth at some scale, then the boundary $(\partial X, \r^\a)$ for an $\a \in (0,1)$ admits a bi-Lipschitz embedding into $\R^n$ for some $n$ \cite{BS}; this is used to apply the Borel density theorem and obtain Theorem \ref{exact-semigroup} for semigroup random walks.
This approach using the Borel density theorem unifies Ledrappier's proof for free groups, where the boundary is totally disconnected, and also Blach\`ere {\it et al.}'s one, where the harmonic measure $\n_\m$ is doubling; in those cases the Borel density theorem holds (e.g.\ \cite[Theorem 1.8]{H}).
In the non-proper case Theorem \ref{acyl}, we need to show a corresponding proposition to Theorem \ref{F}.
This is given in Proposition \ref{pro-non-proper} by employing techniques developed by Maher and Tiozzo \cite{MT}.

\subsection{Applications}

\subsubsection{Kleinian groups}

Let $\H^n$ be the $n$-dimensional real hyperbolic space of constant sectional curvature $-1$ for $n \ge 2$.
The boundary of $\H^n$ is bi-Lipschitz equivalent to the standard sphere $S^{n-1}$ (e.g.\ \cite[III.H.3.19-20]{BH}).
Let $\G$ be a discrete subgroup of isometries of $\H^n$ (a Kleinian group).
There is a natural correspondence between the group of isometries on $\H^n$ and the group of conformal transformations on $S^{n-1}$, and
the dimension formula for a conformal measure on the Riemann sphere $S^2$ is well-studied (e.g.\ \cite{PU}).
The following is a version of it for the harmonic measure.

\begin{corollary}\label{H}
Let $\H^n$ be the $n$-dimensional real hyperbolic space, and $\G$ a discrete subgroup of isometries of $\H^n$.
If a probability measure $\m$ on $\G$ has finite first moment and 
the group $\gr(\m)$ generated by the support of $\m$ is non-elementary,
then the harmonic measure $\n_\m$ on $S^{n-1}$ is exact dimensional and
$$
\dim \n_\m=\frac{h_\m}{l_\m}.
$$
In particular, $\dim \n_\m>0$.
\end{corollary}

\begin{proof}
A discrete subgroup $\G$ of isometries of $\H^n$ is countable and has exponential growth relative to $\H^n$.
Since $\H^n$ has bounded growth at some scale, and the Hausdorff dimension is invariant under a bi-Lipschitz mapping, 
the claim follows directly from Theorem \ref{exact-semigroup}.
\end{proof}

\subsubsection{Continuity of dimension}

The dimension formula implies that the Hausdorff dimension of $\n_\m$ is continuous with respect to $\m$ under some appropriate condition.
We state a claim for a word hyperbolic group endowed with a word metric; then the harmonic measure is defined on the boundary of the Cayley graph.

\begin{corollary}\label{conti}
Let $\G$ be a word hyperbolic group, $\m$ be a probability measure $\m$ which has finite first moment and whose group $\gr(\m)$ generated by the support of $\m$ is non-elementary.
If a sequence of probability measures $\m_i$ converges to $\m$ weakly, i.e., $\m_i(g) \to \m(g)$ for each $g$ in $\G$,
and also $H(\m_i) \to H(\m)$ and $L(\m_i) \to L(\m)$,
then
$\dim \n_{\m_i} \to \dim \n_\m$.
\end{corollary}

\begin{proof}
Since $\gr(\m)$ is a non-elementary subgroup, 
we may assume that $\gr(\m_i)$ is a non-elementary subgroup for all large enough $i$.
Results due to Gou\"ezel, Math\'eus and Maucourant imply that
$l_{\m_i} \to l_\m$ and $h_{\m_i} \to h_\m$ \cite[Proposition 2.3 and Theorem 2.9]{GMM}.
The formula $\dim \n_{\m_i}=h_{\m_i}/l_{\m_i}$ implies that $\dim \n_{\m_i} \to \dim \n_\m$.
\end{proof}

\begin{remark}
The continuity of $h_\m$ and $l_\m$ was first proved by Erschler and Kaimanovich \cite[Theorem 1]{EK}.
A much stronger regularity result is known for probability measures supported on a finite set.
For a hyperbolic group with some metric (e.g.\ a group acting on a hyperbolic space $\H^n$ cocompactly, endowed with the induced metric), or for a non-elementary word hyperbolic group with a word metric,
the entropy and the drift are analytic as functions of step distribution $\m$ (\cite[Corollary 4.2]{GL} and \cite[Theorem 1.1]{G15}).
This implies that $\dim \n_\m$ is also analytic with respect to $\m$ as indicated in \cite{G15}.

\end{remark}

\subsubsection{Brownian motion on a regular covering}

We shall show an exact dimensionality result in a smooth setting.
Let $(M, g)$ be a connected complete Riemannian manifold with bounded sectional curvature. 
Consider the case when $M$ is a regular cover of a Riemannian manifold of finite volume.
A particular case is when $M$ is a universal cover of a compact Riemannian manifold; but in general we assume neither that the covering space is simply connected nor that the quotient manifold is compact.
We denote by $p_t(x, y)$ the (minimal) heat kernel on $M$, i.e.\ a fundamental solution of the heat equation
$(\partial/\partial t)u=\Delta u$, where $\D$ is the Laplacian corresponding to the metric.
We consider continuous analogues of the entropy and the drift, which were introduced by Kaimanovich \cite{K86} and Guivarc'h \cite{G}, respectively. 
Namely, for every $x$ in $M$, let
$$
h_M:=\lim_{t \to \infty}-\frac{1}{t}\int_M p_t(x, y)\log p_t(x, y)dm(y),
$$
$$
l_M:=\lim_{t \to \infty}\frac{1}{t}\int_M d(x, y)p_t(x, y)dm(y),
$$
where $d$ denotes the geodesic distance and $m$ denotes the Riemannian volume measure on $M$ (normalized on a fundamental domain).
In both cases the limit exists and is a constant independent of $x$.
The entropy and the drift are regarded as those of the corresponding Brownian motion.
When the Riemannian manifold $M$ is hyperbolic in the sense of Gromov (Section \ref{hyperbolic}), we denote the boundary by $\partial_\infty M$.
For example, if $M$ is simply connected and the sectional curvature is uniformly bounded in negative values, the boundary $\partial_\infty M$ is topologically a sphere, which admits a $C^\a$-structure for some $\a$ in $(0, 1)$ by Rauch's comparison theorem.
The harmonic measure $\n_M$ is analogously defined for the Brownian motion and is supported on $\partial_\infty M$.
Then we have a dimension formula for this harmonic measure.

\begin{theorem}\label{Bm}
Let $(M, g)$ be a connected complete non-compact Riemannian manifold with bounded sectional curvature. 
If $M$ is hyperbolic and a regular cover of a Riemannian manifold of finite volume, and the covering transformation group $\G$ is non-elementary,
then $h_M$ and $l_M$ are positive and finite, the harmonic measure $\n_M$ associated with the Brownian motion
is exact dimensional and
$$
\dim \n_M=\frac{h_M}{l_M}.
$$
In particular, $\dim \n_M>0$.
\end{theorem}

\begin{proof}
By Furstenberg-Lyons-Sullivan discretization, there is a probability measure $\m$ on the covering transformation group $\G$
such that $\m$ is supported on the whole group $\G$, symmetric, i.e., $\m(g)=\m(g^{-1})$, and has finite first moment with respect to the geodesic distance (\cite{BL} and \cite{K92}).
(See for the cocompact case by Lyons and Sullivan \cite{LS}, who generalized the case of $SL(2, \R)$ by Furstenberg \cite{F}.)
Moreover, the harmonic measure $\n_M$ induced by the Brownian motion starting at $o$ coincides with the one induced by the trajectory of random walk with step distribution $\m$ starting at the same point $o$, and there is a positive constant $T$ such that
$h_\m=T h_M$ and $l_\m=T l_M$ (\cite[Corollaire 3.6]{KL} and \cite[Theorem 2]{K92}).
The group $\G$ is generated by the support of $\m$ as a semigroup,
the condition of lower curvature bound implies that the volume growth of $M$ is at most exponential and that $\G$ has exponential growth relative to $M$, and thus both $h_\m$ and $l_\m$ are positive and finite.
Hence by the result for the random walk with step distribution $\m$ (Theorem \ref{exact} or \ref{exact-semigroup}),
the harmonic measure for the Brownian motion is exact dimensional, and the Hausdorff dimension equals $h_M/l_M$.
\end{proof}

\begin{remark}
This is a generalization of the result stated by Kaimanovich \cite{K90}, and proved by Blach\`ere {\it et al.}, who showed for the universal covering of a compact Riemannian manifold of negative curvature \cite[Theorem 1.9]{BHM11}.
Their proof also uses the dimension formula for a random walk via discretization; but
it requires that the covering be simply connected and the quotient manifold be compact.
Kifer and Ledrappier showed the Hausdorff dimension of harmonic measure (for each starting point) is positive on the boundary of a complete simply connected Riemannian manifold of sectional curvature uniformly bounded in negative values (without group action)\cite[Corollary 3.1]{Kifer-Ledrappier}.
\end{remark}

\subsubsection{Rank one symmetric spaces}

Let us consider a special class of hyperbolic spaces {\it rank one Riemannian symmetric spaces of noncompact type}; up to a compact factor, they are either hyperbolic $n$-spaces over the real, the complex, the quaternions $(n \ge 2)$, or the Cayley plane \cite[Section 19]{M}.
In this case, a homogeneous nature of the space provides a measure $\m$ for which the dimension $\dim \n_\m$ is maximal.

\begin{corollary}
Let $X$ be a rank one Riemannian symmetric space of noncompact type, and $\G$ be a discrete subgroup of isometries of $X$ such that $\G \backslash X$ has finite volume.
Then there exists a probability measure $\m$ on $\G$ supported on the whole $\G$ with finite first moment such that
$$
h_\m=l_\m D_{X},
$$
where $D_X$ denotes the Hausdorff dimension of $(\partial X, \r)$.
\end{corollary}

\begin{proof}
Let us fix a base point $o$ in $X$, then the stabilizer $K_o$ acts transitively on the boundary $\partial X$ which is homeomorphic to a sphere.
The unique $K_o$-invariant probability measure $m_o$ on $\partial X$ (the conformal measure) coincides with the harmonic measure for the Brownian motion starting at $o$.
The conformal measure $m_o$ is the $D_X$-dimensional Hausdorff measure, which is positive on every open set and finite on $\partial X$ (e.g.\ \cite[Th\'eor\`eme 7.7]{C}).
The Furstenberg-Lyons-Sullivan discretization as in the proof of Theorem \ref{Bm} yields a probability measure $\m$ on $\G$ of finite first moment, symmetric with the support $\G$, and the corresponding harmonic measure for random walk $\n_\m$ coincides with $m_o$; in particular $\dim \n_\m=D_X$.
Theorem \ref{exact} implies that $h_\m/l_\m=D_X$.
\end{proof}

\begin{remark}
If $\G$ is cocompact, i.e., the quotient manifold $\G \backslash X$ is compact, then the measure $\m$ is constructed to have a finite exponential moment, whence the above corollary follows from a result in \cite[The proof of Theorem 6.2]{BHM11}.
\end{remark}

\subsection{Organization of the paper}
In Section \ref{preliminary}, we recall some basic facts and known results which we use about Gromov hyperbolic spaces, the Hausdorff dimension of measures and random walks on groups.
In Section \ref{main}, we prove our main results Theorem \ref{exact}, \ref{exact-semigroup} and \ref{acyl}; in order to make it readable depending on various interest, we deal with the proper case and the non-proper one separately, and divide the section into subsections according to the cases.
We also show some additional result specific for word hyperbolic groups.
In Section \ref{questions}, we list some natural questions concerning our results.

\subsection{Notation}
We denote by $C, C', \dots$ constants, and their exact values may change from line to line.
We also denote by say, $C_\d$ a constant depending only on $\d$ if we want to emphasize its dependence on some parameter and independence of all the other parameters.

\section{Preliminary}\label{preliminary}

\subsection{Hyperbolic space}\label{hyperbolic}
We collect some facts about hyperbolic spaces in the sense of Gromov, based on \cite{GdlH}, \cite{Gro} and \cite{BH}.
For non-proper hyperbolic spaces, we refer to \cite{V} and the recent article \cite{DSU}.

Let $(X, d)$ be a metric space.
For $x, y$ and $z$ in $X$, we define the {\it Gromov product} by
$$
(x|y)_z:=\frac{d(z, x)+d(z, y)-d(x,y)}{2}.
$$
We say that $(X, d)$ is $\d$-{\it hyperbolic} if there exists a uniform constant $\d \ge 0$ such that
$$
(x|y)_w \ge \min\{(x|z)_w, (z|y)_w\} - \d
$$
for all $x, y, z$ and $w$ in $X$.
We say that $(X, d)$ is {\it hyperbolic} if it is $\d$-hyperbolic for some $\d$.
If $(X, d)$ is a geodesic metric space, then $(X, d)$ is hyperbolic if and only if there exists a $\d \ge 0$ such that all geodesic triangles are $\d$-{\it slim}, i.e., each side is included in the $\d$-neighbourhood of the other two sides, where $\d$ is possibly a different constant from the previous one.

Let us fix a base point $o$ in $X$.
We say that a sequence $\{x_n\}_{n=0}^\infty$ is {\it divergent} if $(x_n|x_m)_o \to \infty$ as $n, m \to \infty$.
Two divergent sequences $\{x_n\}_{n=0}^\infty$ and $\{y_n\}_{n=0}^\infty$ are equivalent 
if $(x_n|y_m)_o \to \infty$ as $n, m \to \infty$.
Let us define $\partial X$ as a set of equivalence class of divergent sequences, and call it the {\it boundary} of $X$.
For a point $\x$ in $\partial X$, we say that {\it $\{x_n\}_{n=0}^\infty$ converges to $\x$} if it is divergent and its equivalence class is $\x$.
The Gromov product is extended to the boundary by setting
$$
(\x|\y)_o:=\sup\left\{\liminf_{n, m \to \infty} (x_n|y_m)_o\right\},
$$
where the supremum is taken over all the sequences $\{x_n\}_{n=0}^\infty$ and $\{y_n\}_{n=0}^\infty$ converging to $\x$ and $\y$, respectively.
Here we write $(\x|\y)$ for $(\x|\y)_o$ when the base point is $o$ for simplicity of notation.
For all two sequences $\{x_n\}_{n=0}^\infty$ and $\{y_n\}_{n=0}^\infty$ converging $\x$ and $\y$ respectively,
by the $\d$-hyperbolicity, 
we have
\begin{equation}\label{Gpro}
(\x|\y) - 2\d \le \liminf_{n \to \infty}(x_n|y_n) \le (\x|\y).
\end{equation}

We say a metric space $(X, d)$ is {\it geodesic} when arbitrary two points can be joined by a geodesic segment.
If $X$ is a graph, we understand that it is geodesic by identifying each edge with the unit segment $[0,1]$.
For example, a Cayley graph and a complete Riemannian manifold are geodesic.
If $(X, d)$ is a geodesic metric space which is {\it proper}, i.e., every bounded closed ball is compact,
then 
for every $x$ in $X$ and for every boundary point $\x$ in $\partial X$, one can find a geodesic ray starting from $x$ and converging to $\x$
by the Ascoli-Arzel\`a theorem;
and the boundary $\partial X$ is identified with a set of equivalence class of geodesic rays,
where two geodesic rays are equivalent when they are within a bounded distance.
If $(X, d)$ is not proper, then this identification with geodesic rays does not hold since there is no Ascoli-Arzel\`a theorem.
In this case, one can only show that for every $x$ in $X$ and for every $\x$ in $\partial X$, there is a {\it quasi-geodesic ray} starting from $x$ and converging to $\x$.
Although a genuine geodesic ray is not available,
that quasi-geodesic ray can be chosen as a {\it $(1, C_\d)$-quasi-geodesic ray}, 
$\phi: [0, \infty) \to X$, 
\begin{equation}\label{qg}
|t-s|-C_\d \le d(\phi(t), \phi(s)) \le |t-s|+C_\d
\end{equation}
for all $t, s \in [0, \infty)$, where $C_\d$ is a constant independent of a pair $x$ and $\x$, depending only on $\d$ (\cite[Remark 2.16]{KB} and \cite[Section 6]{V}).

Let $(X, d)$ be a hyperbolic geodesic metric space.
We define a {\it quasi-metric}\footnote{In fact, we can define $\r$ on the entire space $X\cup \partial X$ if we allow $\r(x,x)>0$ for $x$ in $X$.}
 $\r$ in $\partial X$ by
$$
\r(x, y):=\exp(-(x|y))
$$
for $x, y$ in $\partial X$.
The quasi-metric $\r$ is non-degenerate: $\r(x, y)=0$ if and only if $x=y$, symmetric: $\r(x, y)=\r(y, x)$ and it satisfies $\r(x, y)\le C(\r(x, z)+\r(z,y))$ for a constant depending only on $\d$.
Consider the space $\partial X$ endowed with the topology induced by $\r$.
This is separable and completely metrizable (\cite[Proposition 3.4.18]{DSU} and \cite[Section 5]{V}).
The space $\partial X$ is in fact compact when $X$ is proper.
Moreover, there exists a metric $\r_\e$ in $\partial X$ with a parameter of range $0 \le \e \le \e_0$ such that 
\begin{equation}\label{quasi-metric}
C_\e^{-1}\r(x, y)^\e \le \r_\e(x, y) \le C_\e\r(x, y)^\e
\end{equation}
for a constant $C_\e$ depending only on $\e$ (\cite[10.- Proposition, Chap.\ 7]{GdlH} and \cite[Section 5]{V}).
Although it would be more natural to use the metric $\r_\e$ than the quasi-metric $\r$,
we will work with the quasi-metric $\r$ for simplicity of notation (to avoid introducing a new parameter). 
Henceforth we define balls in $\partial X$ with respect to the quasi-metric $\r$.

We define a {\it shadow} $S_w(x, R)$ with a base at $w$. For $x$ and $w$ in $X$,
let 
$$
S_w(x, R):=\left\{ \x \in \partial X \ : \ (\x|x)_w \ge d(w, x)- R \right\}.
$$
Informally, a shadow is the set of boundary points such that there exist geodesic rays from $w$ which converge to those points, passing through the ball $B(x, R)$ (with a slightly different constant $R$ by $\d$).
We often take the usual base point $o$ as the base of shadow; in that case, we drop $o$ and write $S(x, R)$.
A shadow is used to control a measure on the boundary $\partial X$.
We have the following comparison between a shadow and a ball.

\begin{lemma}[Proposition 2.1 in \cite{BHM11}]\label{shadow}
For every $\t>0$ there exist positive constants $R_0$ and $C$ such that
for every $R > R_0$, all $x$ in $X$ and all $\x$ in $\partial X$ with $(\x|o)_x \le \t$,
we have
$$
B(\x, (1/C)e^{-|x|+R}) \subset S(x, R) \subset B(\x, Ce^{-|x|+R}),
$$
where $|x|=d(o, x)$ and the ball $B(\x, r)$ is defined by the quasi-metric $\r$.
\end{lemma}

Let $(X, d)$ be a hyperbolic geodesic metric space.
We consider $\G$ a countable group of isometries of $(X, d)$.
The action of $\G$ extends to $X \cup \partial X$ in a natural way, and $\G$ acts on $\partial X$ continuously.
The following lemma provides a control of deformation of balls in $\partial X$ under the $\G$-action.

\begin{lemma}\label{ball}
For every $g$ in $\G$, there exists a constant $c_g > 0$ depending on $g$ such that 
for every $\x$ in $\partial X$ and for every $r \ge 0$,
we have
$$
B(g\x, c_g^{-1}r) \subset gB(\x, r) \subset B(g\x, c_g r).
$$
\end{lemma}

\begin{proof}
We show that 
$\r(g\x, g\y) \le e^{|g|}\r(\x, \y)$ where $|g|=d(o, go)$.
Indeed, let $\{x_n\}_{n=0}^\infty$ and $\{y_n\}_{n=0}^\infty$ be sequences converging to $\x$ and $\y$, respectively.
By the definition of Gromov product and the triangular inequality,
$$
(gx_n|gy_n)_o = (x_n|y_n)_{g^{-1}o} \ge (x_n|y_n)_o - |g|.
$$
Then this implies that
$
(g\x|g\y) \ge (\x|\y)  -|g|.
$
Hence we conclude the claim 
$gB(\x, r)\subset B(g\x, c_gr)$ 
by setting $c_g=e^{|g|}$;
the other inclusion follows by using $g^{-1}$.
\end{proof}

\subsection{Hausdorff dimension of measures}

For a subset $E$ in $\partial X$, we define the Hausdorff dimension with the quasi-metric $\r$ (gauge) instead of the metric $\r_\e$ in (\ref{quasi-metric}).
Since $\r^\e$ and $\r_\e$ are comparable, the Hausdorff dimension with $\r$ is exactly $\e$ times the Hausdorff dimension with $\r_\e$.
We shall recall the definition.
Let 
$$
|E|:=\sup\{\r(x, y) \ : \ x, y \in E\}.
$$
For every $\a \ge 0$ and $\D > 0$, we define
$$
\Hc_\D^\a(E):=\inf \left\{\sum_{i=1}^\infty|E_i|^\a \ : \ \text{$E \subset \bigcup_{i=1}^\infty E_i$ and $|E_i| \le \D$}\right\}.
$$
Then the {\it $\a$-Hausdorff measure} of a set $E$ is
$$
\Hc^\a(E):=\sup_{\D>0}\Hc_\D^\a(E)=\lim_{\D \to 0}\Hc_\D^\a(E).
$$
The {\it Hausdorff dimension} of the set $E$ is
$$
\dim E:=\inf\{\a \ge 0 \ : \ \Hc^\a(E)=0\}=\sup\{\a \ge 0 \ : \ \Hc^\a(E)>0\}.
$$
Let us define the dimension of a measure on $\partial X$.
\begin{definition}
Let $\n$ be a Borel measure on $\partial X$.
The {\it Hausdorff dimension} of $\n$ is defined by
$$
\dim \n:=\inf\{\dim E \ : \ \n(E^{\sf c})=0\},
$$
where $E^{\sf c}$ denotes the complement of the set $E$, i.e., the smallest 
Hausdorff dimension of the support of $\n$.
\end{definition}

We consider a finer dimensional property of measure, the pointwise dimension, which concerns the decreasing rate of $\n(B(\x, r))$ at each point $\x$ in $\partial X$.
The following is used to estimate the dimension of $\n$, and we call it the {\it Frostman-type lemma}.

\begin{lemma}[Frostman-type lemma; e.g.\ \cite{H}, Sect. 8.7]\label{Frostman}
For every finite Borel measure $\n$ on $\partial X$, if there exist $\d_1\ge 0$ and $\d_2 \ge 0$ such that
$$
\d_1 \le \liminf_{r \to 0}\frac{\log \n\left(B(\x, r)\right)}{\log r} \le \d_2 \ \ \text{for $\n$-almost every $\x$},
$$
then 
$
\d_1 \le \dim \n \le \d_2.
$
\end{lemma}
Notice that this Frostman-type lemma also holds for $\partial X$ not necessarily compact; in fact, for every metric space with a finite Borel measure \cite[Corollary 8$\cdot$1]{MSU}.
Lemma \ref{Frostman} implies the following useful characterization for the Hausdorff dimension of $\n$:
\begin{equation}\label{var}
\dim \n=\text{$\n$-$\sup_{\x}$} \liminf_{r \to 0}\frac{\log \n\left(B(\x, r)\right)}{\log r},
\end{equation}
where $\n$-$\sup$ denotes the essential supremum with respect to $\n$.
In other words, $\dim \n$ is the largest possible lower local dimension of $\n$.
We are interested in the case when the above $\liminf$ is a genuine limit and is a constant.

\begin{definition}
We say that a finite Borel measure $\n$ on $\partial X$ is {\it exact dimensional} when the limit
$$
\lim_{r \to 0}\frac{\log \n\left(B(\x, r)\right)}{\log r}
$$
exists $\n$-almost everywhere and it is a constant.
\end{definition}

\subsection{Random walks}

Let $\m$ be a probability measure on $\G$.
Consider an independent sequence of group elements $\{x_n\}_{n=1}^{\infty}$ with common distribution $\m$,
and define $w_n:=x_1 \cdots x_n$ and $w_0=id$.
We denote by $(\O, \Fc, \Pr)$ the probability space where the random walk $\{w_n\}_{n=0}^\infty$ is defined; 
$\O=\G^\N$ is the space of sample paths, $\Fc$ is the Borel $\s$-algebra generated by $\{w_n\}_{n=0}^\infty$ and $\Pr$ is the distribution of sample path starting from $id$, i.e., the image of the product measure $\m^\N$ under the map: $\{x_n\}_{n=1}^\infty \mapsto \{w_n\}_{n=0}^\infty$.

Fix a base point $o$ in $X$, we consider the random trajectory $\{w_no\}_{n=0}^\infty$ in $X$.
Let us consider first the case when the random trajectory converges to a point in the boundary $\partial X$ with the compactified topology in $X\cup \partial X$.
If the random trajectory $\{w_n o\}_{n=0}^\infty$ converges to a (random) point $\x$ in the boundary $\partial X$ for $\Pr$-almost every sample,
then we denote by $\n_\m$ the distribution of $\x$ on $\partial X$. 
We call the distribution $\n_\m$ the {\it harmonic measure}.
More precisely, let us denote by 
$$
\pi : \O \to \partial X
$$
the map defined by
$\o \mapsto \x_\o$, i.e.,
assigning the limiting point in the boundary for each sample (where we write $\o$ for a sample);
then the harmonic measure $\n_\m$ is the pushforward of $\Pr$ by $\pi$.
In general, we write the harmonic measure $\n_{\m, o}$ by referring the starting point $o$, and treat it as a {family} of measures $\{\n_{\m, x}\}_{x \in X}$ parametrized by the starting points.
Once the harmonic measure is defined in this way, then it is {\it stationary}, i.e.,
\begin{equation}\label{stationary}
\n_\m=\sum_{g \in \G}\m(g)g\n_\m,
\end{equation}
where $g\n_\m$ denotes the image of $\n_\m$ by the action of $g$.
(Here $g\n_{\m}$ is written as $\n_{\m, go}$ when $\n_\m=\n_{\m, o}$.)
Furthermore the harmonic measure is {\it ergodic} in the sense that every $\G$-invariant set has the measure $0$ or $1$; more precisely:

\begin{theorem}\label{ergodic}
Assume that the harmonic measures are defined for $\m$ and its reflected measure $\check \m$, i.e., $\check \m(g)=\m(g^{-1})$.
Then the product measure $\n_\m \otimes \n_{\check \m}$ is ergodic under the diagonal action of $\G$ on $\partial X \times \partial X$.
In particular, the harmonic measure $\n_\m$ is ergodic under the $\G$-action on $\partial X$, i.e.,
for every Borel set $A$ in $\partial X$, if $A$ is $\G$-invariant, then
$\n_\m(A)=0$ or $1$. 
\end{theorem}

\begin{proof}
The proof is given in \cite[Theorem 6.3]{K00}; here we reproduce it for the sake of convenience.
We consider bilateral random walks. 
Define the bilateral path $\{w_n\}_{n \in \Z}$ corresponding to the bi-infinite sequence $\{x_n\}_{n \in \Z}$ of independent identically distributed increments with the law $\m$ by
\begin{equation*}
w_n=
\begin{cases}
x_1 \cdots x_n	\ &\text{if $n \ge 1$}\\
id 				\ &\text{if $n=0$}\\
x_0^{-1}\cdots x_{n+1}^{-1} \ &\text{if $n \le -1$}.
\end{cases}
\end{equation*}
In other words, in negative indices, we consider the random walk with reflected step distribution $\check \m$ starting from $id$. 
We call the random walk in positive indices the {\it forward} random walk and the one in negative indices the {\it backward} one.
We denote by $(\overline\O, \overline\Pr)$ the space of bilateral path with the distribution given by pushforward of $\m^\Z$ via the map $\{x_n\}_{n \in \Z} \mapsto \{w_n\}_{n \in \Z}$.
Let us define the shift $T$ by
$T(\{w_n\}_{n \in \Z})=\{w_1^{-1}w_{n+1}\}_{n \in \Z}$. 
Note that the shift $T$ is induced by the bilateral Bernoulli shift in the space of increments, and is measure preserving and ergodic.
Let us consider the map $\overline{\pi} : \overline \O \to \partial X \times \partial X$ defined by $\overline \o \mapsto (\x, \check \x)$ for $\overline \Pr$-almost every sample $\overline \o$ where $\x$ is the limiting point of forward random walk and $\check \x$ is the one of backward random walk.
Define the diagonal action of $\G$ on $\partial X \times \partial X$.
Then the product measure $\n_\m \otimes \n_{\check \m}$ on $\partial X \times \partial X$ is ergodic under the diagonal action of $\G$; 
indeed, if there exists a $\G$-invariant Borel set $A$ such that $0<\n_\m \otimes \n_{\check \m}(A)<1$, then $\overline{\pi}^{-1}(A)$ is the shift $T$-invariant in $\overline \O$, and since $(\overline \O, \overline \Pr, T)$ is ergodic, this implies that $\Pr(\overline{\pi}^{-1}(A))=0$ or $1$; a contradiction. 
To see the ergodicity of a single $\n_\m$, take $A \times \partial X$ for a $\G$-invariant $A$ in $\partial X$.
\end{proof}

Let us consider then when the harmonic measure is defined on $\partial X$, or when the map $\pi: \O \to \partial X$ is defined in such a way.
Recall that a group $\G$ is {\it non-elementary} if the group acts on $X$ by isometries and the orbit of a (or, equivalently every) point has infinitely many limit points in the boundary $\partial X$.
This is equivalent to say that
it contains at least two hyperbolic elements with disjoints fixed points in the boundary $\partial X$, 
and also that $\G$ acts on $\partial X$ without any fixed points \cite[8.2]{Gro}. 
Recall that $\rm gr(\m)$ (resp.\ $\rm sgr(\m)$) is the group (resp.\ the semigroup) generated by the support of $\m$.
They naturally coincide when $\m$ is symmetric, i.e., $\m(g)=\m(g^{-1})$, or when 
the semigroup generated by the support of $\m$ is a group;
but in general, we assume neither the one nor the other.
The following is a special case of a theorem in \cite{K00}.

\begin{theorem}[\cite{K00}, Theorem 2.4 and Sect.\ 7]\label{Kconvergence}
Let $(X, d)$ be a hyperbolic proper geodesic space, and $\G$ be a countable group of isometries acting on $X$.
For every probability measure $\m$ on $\G$, if the group $\rm gr(\m)$ generated by its support is non-elementary,
then the random trajectory $\{w_n o\}_{n=0}^\infty$ converges to a (random) point $\x$ in the boundary $\partial X$ almost surely in the compactified topology in $X\cup \partial X$.
\end{theorem}

In Theorem \ref{Kconvergence}, if we impose a stronger assumption, the random trajectory $\{w_n o\}_{n=0}^\infty$ goes almost along a geodesic ray, which we call the {\it geodesic tracking property}.
Assume that $\G$ has exponential growth relative to $X$.
If the group $\gr(\m)$ is non-elementary and $\m$ has finite first moment, the trajectory tracks a geodesic in a sublinear deviation.

\begin{theorem}[\cite{K00}, Sect.\ 7]\label{tracking}
Let $(X, d)$ be as in Theorem \ref{Kconvergence}.
Let $\m$ be a probability measure such that the group $\rm gr(\m)$ generated by its support is non-elementary and has finite first moment.
If $\G$ has exponential growth relative to $X$,
then the drift $l_\m$ is positive, and 
for almost every sample path, 
there exists a unit speed geodesic ray $\x$ such that
$$
\lim_{n \to \infty}\frac{1}{n}d(w_no, \x(l_\m n))=0.
$$
\end{theorem}

\begin{proof}[A sketch of proof]
If $\gr(\m)$ is non-elementary, then the boundary $\partial X$ endowed with the harmonic measures $\{\n_{\m, x}\}_{x \in X}$ is a non-trivial $\m$-boundary.
When $\m$ has finite first moment and $\G$ has exponential growth relative to $X$, this implies that the corresponding Poisson boundary is non-trivial and $l_\m>0$.
Then for every $\e>0$, we have $(w_no|w_{n+1}o)_o \ge (l_\m -\e)n$ for all large enough $n$, and $\{w_no\}_{n=0}^\infty$ is a Cauchy sequence with respect to the quasi-metric $\r$.
Hence there exists a point $\x$ to which $\{w_no\}_{n=0}^\infty$ converges, and a geodesic emanating from $o$ towards $\x$ satisfies the claim.
\end{proof}

Concerning the entropy, we use the Shannon theorem for random walks due to Kaimanovich and Vershik \cite{KV}, and Derriennic \cite{Derriennic} (who attributes an observation due to J.P.\ Conze).

\begin{theorem}[\cite{KV}, Theorem 2.1, and \cite{Derriennic}, Sect.\ IV]\label{Shannon}
For every countable group $\G$, and every probability measure $\m$ on $\G$ with $H(\m)<\infty$, 
we have
$$
\lim_{n \to \infty}-\frac{1}{n}\log \m^{\ast n}(w_n) = h_\m
$$
for $\Pr$-almost every sample.
\end{theorem}

\section{Dimension formula: Proof of main results}\label{main}

We write the harmonic measure $\n$, the entropy $h$ and the drift $l$ without specifying the step distribution $\m$ if there is no possibility of confusion.

\subsection{The upper bound.}

The following is an upper bound of pointwise dimension of the harmonic measure due to Le Prince \cite[Proposition 2.3]{LP07}.
In order to compare the method with the one for the lower bound, we briefly outline the proof.
We show it in a slightly extended form (but the proof is essentially the same).

\begin{theorem}\label{upperbound}
Let $(X, d)$ be a hyperbolic (not necessarily proper or geodesic) metric space, and $\G$ a countable group of isometries of $X$.
If a probability measure $\m$ on $\G$ for which the harmonic measure $\n_\m$ is defined has finite entropy, i.e., $H(\m)<\infty$, and
$$
l_\m:=\inf_\Pr \liminf_{n \to \infty}\frac{(w_{n+1}o|w_n o)}{n} >0,
$$
then we have
$$
\limsup_{r \to 0}\frac{\log \n_\m\left(B(\x, r)\right)}{\log r} \le \frac{h_\m}{l_\m}
$$
for $\n_\m$-almost every $\x$ in $\partial X$. 
Here we denote by $\inf_\Pr$ the essential infimum with respect to the measure $\Pr$ on the space of samples $\O$.
In particular, we have
$$
\dim \n_\m \le \frac{h_\m}{l_\m},
$$
and if $\m$ has finite first moment, then $l_\m$ coincides with the drift.
\end{theorem}

\begin{remark}
Note that here we do not assume that $\G$ has exponential growth relative to $X$.
\end{remark}

\begin{proof}[Proof of Theorem \ref{upperbound}]
For every positive $\l$ less than $l$, we define the event
$$
A_{\e, N}:=\left\{\o \in \O \ : \ \text{$(w_{n+1}o|w_n o) \ge (\l-\e)n$ and $\m^{\ast n}(w_n) \ge e^{-n(h+\e)}$ for all $n \ge N$}\right\}.
$$
For every $\e>0$, there exists an $N_\e$ such that $\Pr(A_{\e, N_\e}) \ge 1-\e$ by the Shannon theorem for random walks (Theorem \ref{Shannon}).
Let $A_\e:=A_{\e, N_\e}$.
For a sample $z=\{z_n\}_{n=0}^\infty$ in $\O$, we denote by $C_n(z)$ the samples whose $n$-th component is $z_n$.
Then for $\Pr$-almost every $z$ in $A_\e$, the following ratio has a positive limit
$$
\lim_{n \to \infty}\frac{\Pr\left(A_\e \cap C_n(z)\right)}{\Pr \left(C_n(z)\right)}>0.
$$
(See e.g.\ \cite[Lemma 4.18]{L01} or \cite[(1.4.5) in the proof of Theorem 1.4.1]{K98}.)
Since there exists a constant $C>0$ such that
$$
A_\e \cap C_n(z) \subset \left\{ \o \in \O \ : \ \x_\o \in B\left(\x_z, Ce^{-n(\l-\e)}\right)\right\}
$$
for all $n \ge N_\e$ \cite[Lemma 2.2]{LP07},
we obtain
$$
\limsup_{n \to \infty}\frac{\log\n\left(B\left(\x_z, Ce^{-n(\l-\e)}\right)\right)}{-n(\l-\e)} \le \frac{h+\e}{\l-\e}
$$
for $\Pr$-almost every $z$ in $A_\e$, where we use $\Pr(C_n(z)) =\m^{\ast n}(z_n) \ge e^{-n(h+\e)}$ for all $n \ge N_\e$.
The event $A_\e$ has $\Pr$-measure at least $1-\e$, and the desired upper bound follows by taking $\l \to l$.
The rest follows from the Frostman-type lemma (Lemma \ref{Frostman}).

If $\m$ has finite first moment, then $d(w_n o, w_{n+1}o)=o(n)$ almost surely, and this implies that
$(w_n o | w_{n+1} o)/n$ converges to the drift almost surely.
\end{proof}

\subsection{The proper case}

In this section we prove Theorem \ref{exact} and \ref{exact-semigroup}.
The lower bound of pointwise dimension follows from an estimate on a restricted set of $\n_\m$-measure positive.

\begin{theorem}\label{F}
Let $(X, d)$ be a hyperbolic proper geodesic metric space, and $\G$ a countable group of isometries of $X$ having exponential growth relative to $X$.
If the group $\gr(\m)$ generated by the support of $\m$ is non-elementary and $\m$ has finite first moment,
then for every $\e > 0$, there exists a Borel set $F_\e$ in $\partial X$ such that $\n_\m(F_\e)\ge 1-2\e$ and
$$
\liminf_{r \to 0}\frac{\log \n_\m\left(F_\e \cap B(\x, r)\right)}{\log r}\ge \frac{h_\m}{l_\m} -\e
$$
for $\n_\m$-almost every $\x$ in $\partial X$.
\end{theorem}

First we shall show Theorem \ref{exact} and \ref{exact-semigroup} by using Theorem \ref{F}.

\begin{proof}[Proof of Theorem \ref{exact}]
For the set $F_\e$ in Theorem \ref{F}, 
we consider the restriction of $\n$ on $F_\e$ and denote it by $\n |_{F_\e}$.
By definition, we have $\dim \n \ge \dim \n |_{F_\e}$, and the last term is at least $h/l -\e$ by Theorem \ref{F} and Lemma \ref{Frostman}.
By the characterization of dimension $\dim \n$ (\ref{var}), 
the set 
$$
G_\e:=\left\{ \x \in \partial X \ : \ \liminf_{r \to 0}\frac{\log \n\left(B(\x, r)\right)}{\log r}\ge \frac{h}{l} -2\e\right\}
$$
has a $\n$-positive measure.
We will prove that $G_\e$ is $\G$-invariant.
We assume that the support of $\m$ generates the whole group $\G$ as a semigroup; otherwise the same proof applies to the non-elementary subgroup which the support generates.
Since the harmonic measure $\n$ is $\m$-stationary, and the support of $\m$ generates $\G$ as a semigroup, by the stationarity of harmonic measure (\ref{stationary})
for every $g$ in $\G$ there exists a constant $c_{g, \m}>0$ depending on $g$ and $\m$ such that
$$
c_{g, \m}^{-1}\n \le g^{-1}\n \le c_{g, \m}\n.
$$
By Lemma \ref{ball}, we have for every $\x$ in $\G$,
$$
\n(B(g\x, r)) \le \n (gB(\x, c_g r))=g^{-1}\n (B(\x, c_g r)) \le c_{g, \m}\n (B(\x, c_g r)).
$$
Then $gG_\e \subset G_\e$ for every $g$ in $\G$, and $g^{-1}G_\e \subset G_\e$; hence $G_\e$ is $\G$-invariant.
Since $\n$ is ergodic with respect to the $\G$-action and $\n(G_\e)>0$, we have $\n(G_\e)=1$ for every $\e>0$.
Therefore
$$
\liminf_{r \to 0}\frac{\log \n\left(B(\x, r)\right)}{\log r}\ge \frac{h}{l}
$$
for $\n$-almost every $\x$ in $\partial X$.
The upperbound
$$
\limsup_{r \to 0}\frac{\log \n\left(B(\x, r)\right)}{\log r}\le \frac{h}{l}
$$
for $\n$-almost every $\x$ in $\partial X$, 
follows from Theorem \ref{upperbound}.
Thus $\n$ is exact dimensional and the local dimension equals $h/l$ for $\n$-almost every point.
The second assertion follows from the Frostman-type lemma (Lemma \ref{Frostman}).
\end{proof}

\begin{proof}[Proof of Theorem \ref{exact-semigroup}]
If $(X, d)$ has bounded growth at some scale, then the boundary $(\partial X, \r)$ has a finite Assouad dimension by Bonk and Schramm \cite[Theorem 9.2]{BS}, and this implies that for each fixed $\a \in (0, 1)$, there exists some integer $n$ such that $(\partial X, \r^\a)$ admits a bi-Lipschitz embedding into $\R^n$ (\cite[2.6.\ Proposition]{A}; see \cite{BS} for a further discussion of the Assouad dimension in this context), i.e.,
there exists a map $f: \partial X \to \R^n$ satisfying that there is a constant $L>0$ such that 
\begin{equation}\label{biL}
(1/L)\r(\x, \y)^\a \le \|f(\x)-f(\y)\|_{\R^n} \le L \r(\x, \y)^\a
\end{equation}
for all $\x, \y$ in $\partial X$.
Let $f_\ast \n$ be the pushforward of $\n$ by $f$, and $F:=F_\e$ be the set as in Theorem \ref{F}.
The Borel density lemma for the Borel measure $f_\ast \n$ in $\R^n$ (e.g.\ \cite[Theorem 8.5.4]{PU}) implies that
$$
\lim_{r \to 0}\frac{f_\ast \n \left(f(F)\cap B_{\R^n}(f(\x), r)\right)}{f_\ast \n\left(B_{\R^n}(f(\x), r)\right)}=1 \ \text{for $\n$-almost every $\x$ in $F$}.
$$
(Here we denote by $B_{\R^n}$ and just by $B$ balls in $\R^n$ and in $\partial X$ respectively.)
Since $f$ satisfies (\ref{biL}),
$$
\liminf_{r \to 0}\frac{\n\left(F\cap B(\x, (Lr)^{1/\a})\right)}{\n\left(B(\x, (r/L)^{1/\a})\right)} \ge 1 \ \text{for $\n$-almost every $\x$ in $F$}.
$$
For $\n$-almost every $\x$ in $F$, there exist positive constants $c(\x)>0$ and $r(\x)>0$ such that for all $0< r < r(\x)$,
$$
\n\left(F\cap B(\x, L^{2/\a}r)\right) \ge c(\x) \n\left(B(\x, r)\right).
$$
By Theorem \ref{F}, we have
\begin{equation}\label{eq}
\liminf_{r \to 0}\frac{\log \n\left(B(\x, r)\right)}{\log r} \ge \frac{h}{l}-\e \ \text{for $\n$-almost every $\x$ in $F$}.
\end{equation}
Since $F=F_\e$ and $\n(F_\e) \ge 1-2\e$, 
we conclude that
for $\n$-almost every $\x$ in $\partial X$,
$$
\liminf_{r \to 0}\frac{\log \n\left(B(\x, r)\right)}{\log r} \ge \frac{h}{l}.
$$
The rest follows as in the proof of Theorem \ref{exact}.
\end{proof}

\begin{remark}
In view of Theorem \ref{exact-semigroup}, one might wonder as to whether Theorem \ref{exact} holds for a semigroup random walk as well on a general proper hyperbolic space.
But at this stage, we are not aware of a proof which covers this level of generality, or of any counter examples.
\end{remark}

By the geodesic tracking property (Theorem \ref{tracking}), 
for $\Pr$-almost every sample $\o$ in $\O$, there exists a geodesic ray $\x_\o$ issuing from $o$ such that
$$
d(w_n(\o)o, \x_\o(l n))=o(n).
$$
Moreover, one can assign the (unit speed) geodesic ray $\x_\o$ in the measurable way, i.e.,
the map from $\O$ to the space of geodesic rays issuing from $o$: $\o \mapsto \x_\o$ is Borel measurable, where the target is endowed with the convergence on closed bounded sets topology.
Indeed, the measurable section theorem states the following:
\begin{theorem}[Theorem 3.4.1 in \cite{Arveson}]\label{measurable section}
Let $P$ be a Polish space, let $Y$ be a Borel space, and let $f$ be a function from $P$ onto $Y$ satisfying that
\begin{itemize}
\item[(i)]	$f$ maps open sets to Borel sets;
\item[(ii)] the inverse image of each point of $Y$ is a closed subset of $P$.
\end{itemize}
Then $f$ has a Borel section, i.e.,
there exists a Borel measurable map $g$ from $Y$ to $P$ such that $f \circ g=id_Y$.
\end{theorem}

The space of geodesic rays issuing from $o$ admits a structure of Polish space, and $\partial X$ is a Borel space;
one can check that the map assigning the end point of a geodesic ray satisfies (i) and (ii) in the measurable section theorem (Theorem \ref{measurable section}),
and it has a Borel section.
Composing the Borel section with the map $\O$ from $\partial X$ given by the limiting point of the random walk,
we obtain the desired Borel map from $\O$ to the space of geodesic rays.

Let us define the event 
$$
A_{\e, N}:=\left\{ \o \in \O \ : \ \text{$d(w_n(\o)o, \x_\o(l n)) \le \e n$ and $\m^{\ast n}\left(w_n(\o)\right) \le e^{-n(h-\e)}$ for all $n \ge N$}\right\}.
$$
Notice that $A_{\e, N}$ is increasing in $N$ and $\Pr\left(\bigcup_N A_{\e, N}\right)=1$ by Theorem \ref{tracking} and Theorem \ref{Shannon}; 
hence for every $\e>0$ there exists an $N_\e$ such that for all $N \ge N_\e$ we have
$\Pr(A_{\e, N}) \ge 1 -\e$.
Let $A_\e:=A_{\e, N_\e}$.
The trajectories on the event $A_{\e, N}$ are well-controlled by geodesic rays.
The following lemma implies that
on the event $A_{\e, N}$,
if two limiting points are in the same shadow around $x$ at the distance $(l\pm\e)n$ from $o$,
then the trajectories at time $n \ge N$ are in a distance $2 \e n$ (up to an additive constant).

\begin{lemma}\label{hit}
Fix an $R>4\d$. 
For every $n \ge N_\e$, let $x$ be a point in $X$ such that $\left||x| - ln\right| \le \e n$.
For every $\o$ in $A_{\e}$,
if the limiting point of $\{w_n(\o)o\}_{n=0}^\infty$ is in the shadow $S(x, R)$,
then $w_n(\o)o$ is in the ball $B(x, 2\e n+C_{R, \d})$ for the $n \ge N_\e$, where $C_{R, \d}$ is a constant depending only on $R$ and $\d$.
\end{lemma}

\begin{proof}
Let $\x_\o$ be a geodesic ray from $o$ satisfying that $d(w_n(\o)o, \x_\o(ln)) \le \e n$ for all $n \ge N$.
If the limiting point of $\{w_n(\o)o\}_{n=0}^\infty$ is in the shadow $S(x, R)$, then the limiting point of $\x_\o$ is also in the same shadow, whence
the geodesic ray $\x_\o$ intersects the ball $B(x, R+10\d)$ by the $\d$-hyperbolicity.
By the triangular inequality, we have $d(x, \x_\o(|x|))\le 2R+20\d$ (the geodesic ray has unit speed), and the claim follows.
\end{proof}

\begin{remark}\label{remarkhit}
The above Lemma \ref{hit} is also proved with $(1, C_\d)$-quasi-geodesic rays (\ref{qg}) in a non-proper space.
Namely, one can replace the geodesic ray $\x_\o$ by a $(1, C_\d)$-quasi-geodesic ray in the definition of $A_\e$, and show Lemma \ref{hit} (with a different constant $C_{R, \d}$) by $\d$-hyperbolicity.
This will be used in Section \ref{non-proper}.
\end{remark}

Recall that the map $\pi : \O \to \partial X$ is defined by
$\o \mapsto \x_\o$, i.e.,
assigning the limiting point in the boundary for $\Pr$-almost every sample.
We define $\Sc$ as the smallest $\s$-algebra for which the map $\pi$ is measurable.
We consider a system of conditional probability measures of $\Pr$ with respect to the sub $\s$-algebra $\Sc$ of $\Fc$ on $\O$.
Recall that the probability space $(\O, \Fc, \Pr)$ is {\it standard}, i.e., $\O=\G^\N$ admits a structure of complete separable metric space, 
$\Fc$ is the Borel $\s$-algebra and $\Pr$ is a Borel probability measure.
Then there exists a system of Borel probability measures $\{\Pr_{\pi(\o)}\}_{\o \in \O}$ the {\it conditional probability measures} of $\Sc$ such that
$$
\Pr=\int_\O \Pr_{\pi(\o)}d\Pr(\o).
$$
Note that for every $\Fc$-measurable set $A$, the map $\o \mapsto \Pr_{\pi(\o)}(A)$ is $\Sc$-measurable, and the map is determined only for $\Pr$-almost every $\o$.
(See the original work by Rohlin \cite{Rohlin} or a modern treatment in \cite{Simmons}.)
Since the harmonic measure $\n$ is the image measure of $\Pr$ by $\pi$,
for every Borel set $A$ in $\partial X$ and for every Borel set $B$ in $\O$, we have
$$
\Pr( B | \pi^{-1}(A))=\frac{1}{\n(A)}\int_A \Pr_\x(B)d\n(\x).
$$

Let us define a Borel set in the boundary $\partial X$ by
\begin{equation}\label{defF}
F_\e:=\left\{ \x \in \partial X \ : \ \Pr_\x(A_\e) \ge \e \right\}.
\end{equation}
Since it holds that
\begin{align*}
1-\e \le \Pr(A_\e)=\int_{\partial X} \Pr_\x(A_\e)d\n(\x)	&=\int_{F_\e}\Pr_\x(A_\e)d\n(\x)+\int_{F_\e^{\sf c}}\Pr_\x(A_\e)d\n(\x)\\
												&\le \n(F_\e)+\e \n(F_\e^{\sf c}),
\end{align*}
we have 
\begin{equation}\label{Fe}
\n(F_\e) \ge 1-2 \e.
\end{equation}

\begin{proof}[Proof of Theorem \ref{F}]
Let $N \ge N_\e$, and
let $\{x_n\}_{n=0}^\infty$ be an arbitrary sequence in $X$ such that $|x_n| = ln$ and $x_n$ is within a distance $\e n$ to the orbit $\G o$ for all $n \ge N$.
We will fix the sequence $\{x_n\}_{n=0}^\infty$ later.
By Lemma \ref{hit}, for every $\o$ in $A_\e$ and for every $n \ge N$,
if the limiting point $\x_\o$ is in $S(x_n, R)$, then the trajectory $\{w_n(\o)o\}_{n=0}^\infty$ must hit the ball $B(x_n, 2\e n+C_0)$ at time $n$, where $C_0:=C_{R,\d}$.
Hence for every $n \ge N$, we have
\begin{align*}
&\Pr\left(\x_\o \in F_\e \cap S(x_n, R)\right) \\
&\le \Pr\left(A_\e \cap\left\{w_n(\o)o \in B(x_n, 2\e n+C_0)\right\}\right) + \Pr\left(A_\e^{\sf c}\cap\{\x_\o \in F_\e \cap S(x_n, R)\}\right).
\end{align*}
The second term in the right hand side has a bound by using the conditional probability measures:
\begin{align*}
\Pr\left(A_\e^{\sf c}\cap\{\x_\o \in F_\e \cap S(x_n, R)\}\right)
&=\int_{F_\e \cap S(x_n, R)}\Pr_{\x}\left(A_\e^{\sf c}\right)d\n(\x) \\
&\le (1-\e) \n\left(F_\e \cap S(x_n, R)\right),
\end{align*}
where in the last inequality we use $\Pr_\x(A_\e^{\sf c}) \le 1-\e$ for $\n$-almost every $\x$ in $F_\e$ by (\ref{defF}).
Therefore for all $n \ge N$ we obtain
$$
\e \n\left(F_\e \cap S(x_n, R)\right) \le \Pr\left(A_\e \cap\left\{w_n(\o)o \in B(x_n, 2\e n+C_0)\right\}\right).
$$
Recall that $x_n$ is within a distance $\e n$ to the orbit $\G o$ for all $n \ge N$. In this case, since the group $\G$ has exponential growth relative to $X$,
the number of group elements $g$ in $\G$ such that $go$ is in the ball $B(x_n, 2\e n+C_0)$ is at most $Ce^{C(3\e n+C_0)}$ by the triangular inequality.
Therefore we have 
$$
\Pr\left(A_\e \cap\left\{w_n(\o)o \in B(x_n, 2\e n+C_0)\right\}\right) \le Ce^{C(3\e n+C_0)}e^{-n(h-\e)}.
$$

For every $\o$ in $A_{\e, N}$, we take the limiting point $\x_\o$, which also denotes a (unit speed) geodesic ray issuing from $o$ towards the point, we set $x_n:=\x_\o(ln)$ since it is within a distance $\e n$ to the orbit $\G o$ for all $n \ge N$.
Then for every $\o$ in $A_{\e, N}$, we obtain
\begin{align*}
\liminf_{n \to \infty}\frac{\log \n\left(F_\e \cap S(\x_\o(ln), R)\right)}{-n} \ge h-\e -3C\e.
\end{align*}
By Lemma \ref{shadow}, we have
\begin{align*}
\liminf_{r \to 0}\frac{\log \n\left(F_\e \cap B(\x_\o, r)\right)}{\log r} \ge \frac{h}{l}-C'\e.
\end{align*}
This holds for every $\o$ in $A_{\e, N}$ for all $N \ge N_\e$, and thus for $\Pr$-almost every $\o$; we conclude the proof.
\end{proof}

\subsection{The case of word hyperbolic groups}

Let $\G$ be a word hyperbolic group, i.e., $\G$ is a finitely generated group, and the Cayley graph associated with a finite symmetric set of generators (the group endowed with a word metric) is hyperbolic.
We consider in the group $\G$ a metric $d$ which is left invariant, hyperbolic and quasi-isometric to some word metric.
Since a geodesic space which is quasi-isometric to a hyperbolic geodesic space is also hyperbolic (with a different constant $\d$ for the $\d$-hyperbolicity), the metric $d$ in $\G$ is in fact quasi-isometric to any word metric \cite{GdlH}[14.-Corollaire, Chap.\ 5].
Here we do not assume that the metric $d$ is geodesic, but we assume that it is hyperbolic.
It is of interest the case when the group admits a (non-geodesic) hyperbolic metric other than a word metric (but quasi-isometric to it), we prove the following theorem.

\begin{theorem}\label{word}
Let $\G$ be a word hyperbolic group, and $d$ be a metric which is left invariant, hyperbolic and quasi-isometric to some word metric.
We denote by $\partial \G$ the boundary associated with the hyperbolic metric space $(\G, d)$.
For every probability measure $\m$ on $\G$, if the group $\gr(\m)$ generated by the support of $\m$ is non-elementary,
then the harmonic measure $\n_\m$ on $\partial \G$ is exact dimensional, and 
$$
\dim \n_\m=\frac{h_\m}{l_\m}.
$$
\end{theorem}

\begin{proof}
Fix the identity element as a base point $o$.
Since $(\G, d)$ is quasi-isometric to a hyperbolic geodesic space (a Cayley graph), a geodesic ray $\x$ in the Cayley graph is a quasi-geodesic ray in $(\G, d)$.
Note that for all such quasi-geodesic rays $\x$, 
Lemma \ref{shadow} holds for uniform constants; and also that $\G$ has exponential growth relative to $(\G, d)$ since it is the case relative to the Cayley graph.
A Cayley graph has bounded growth at some scale, whence the proof follows as in Theorem \ref{exact-semigroup}, \ref{F} and \ref{upperbound}.
We omit the details.
\end{proof}

\subsection{The non-proper case: acylindrical actions}\label{non-proper}

We extend our results to a group acting on a non-proper Gromov hyperbolic space.
Let $X$ be a separable geodesic (but not necessarily proper) Gromov hyperbolic space.
In \cite{MT}, Maher and Tiozzo showed the following:

\begin{theorem}[Theorem 1.1 in \cite{MT}]\label{MTconv}
Let $\G$ be a countable group of isometries of a separable geodesic Gromov hyperbolic space $X$.
If the semigroup ${\rm sgr}(\m)$ generated by the support of $\m$ is a non-elementary subgroup of $\G$,
then the harmonic measure is well-defined; namely,
for the random walk $\{w_n\}_{n=0}^\infty$ with the step distribution $\m$, 
almost every trajectory $\{w_n o\}_{n=0}^\infty$ converges to a point $\x$ in the boundary $\partial X$,
and the distribution $\n$ of $\x$ is non-atomic and the unique $\m$-stationary measure on $\partial X$.
\end{theorem}

Moreover, when the probability measure $\m$ has finite first moment, they showed the positivity of drift and the geodesic tracking property by using the above result.

\begin{theorem}[Theorem 1.2 and 1.3 in \cite{MT}]\label{MTgeodesic}
If in addition to the condition in Theorem \ref{MTconv} the probability measure $\m$ has finite first moment, then there is a positive constant $l_\m>0$ such that for almost every sample path, we have
$$
\lim_{n \to \infty}\frac{1}{n}d(o, w_n o)=l_\m.
$$
Furthermore, for almost every sample path, there exists a quasi-geodesic ray $\g$ such that
$$
\lim_{n \to \infty}\frac{1}{n}d(w_n o, \g)=0.
$$
\end{theorem}

Recall that a group $\G$ acts on $X$ {\it acylindrically} if the group acts on $X$ by isometries and satisfies the following:
for every $K \ge 0$, there exist constants $R, N$ such that for all two points $x$ and $y$ in $X$ with $d(x, y) \ge R$ 
there are at most $N$ group elements $g$ in $\G$ satisfying that $d(x, g x) \le K$ and $d(y, g y) \le K$.

Following \cite{MT}, we define a group element of {\it bounded geometry}:
Let $v$ be a group element in $\G$, and $K, R \ge 0$. 
For two boundary points $\a$ and $\b$ in $\partial X$,
we say that a group element $g$ has $(K, R, v)$-{\it bounded geometry} with respect to the pair $(\a, \b)$ if
\begin{itemize}
\item[(i)] $d(go, gvo) \ge R$, 
\item[(ii)] $\a \in S_{gvo}(go, K)$ and $\b \in S_{go}(gvo, K)$.
\end{itemize}
We denote by $\Oc_{K, R, v}(\a, \b)$ the set of $(K, R, v)$-bounded geometry elements for $(\a, \b)$.
Note that the set is $\G$-equivariant in the sense that $g\Oc_{K, R, v}(\a, \b)=\Oc_{K, R, v}(g\a, g\b)$ for every group element $g$ in $\G$.
The number of bounded geometry elements has linear growth in the following sense:

\begin{lemma}[Proposition 6.2 in \cite{MT}]
There exists $K_0$ such that for every $K \ge K_0$, there exists $R_0$ such that for every $R \ge R_0$, there exists a constant $C$ such that 
for every $\a, \b$ in $\partial X$, every $r>0$ and every group element $v$ in $\G$, we have
$$
\sharp\{g \in \G \ : \ go \in \Oc_{K, R, v}(\a, \b)o \cap B(o, r)\} \le Cr.
$$
\end{lemma}
\begin{remark}
In \cite{MT}, they estimate the number of orbit points $go$ in the set $\Oc_{K, R, v}(\a, \b)o \cap B(o, r)$; but the proof actually yields the same bound for the number of group elements.
\end{remark}

Let us consider bilateral random walks as in the proof of Theorem \ref{ergodic}. 
Recall that $(\overline\O, \overline\Pr, \{w_n\}_{n \in \Z}, T)$ is the probability space of bilateral paths with the shift $T$ defined by
$T(\{w_n\}_{n \in \Z})=\{w_1^{-1}w_{n+1}\}_{n \in \Z}$. 
We denote by $\x_+$ (resp. $\x_-$) the limit point of forward (resp. backward) random walk, and the distribution by $\n_\m$ (resp. $\n_{\check \m}$).
By an appropriate choice of $K$, $R$ and $v$, the set $\Oc_{K, R, v}(\a, \b)$ is non-empty and has linear growth for $\n_\m \otimes \n_{\check \m}$-almost every $(\a, \b)$.

\begin{proposition}[Proposition 6.4 in \cite{MT}]\label{MTO}
There exist constants $K$, $R$ and a group element $v$ in $\G$ such that 
the set $\Oc_{K, R, v}(\a, \b)$
of bounded geometry elements is non-empty (in fact, infinite) and has linear growth for $\n_\m \otimes \n_{\check \m}$-almost every $(\a, \b)$.
\end{proposition}

Let us denote the above set of bounded geometry elements by $\Oc(\o)$ for a (bilateral) sample $\o$.
The following is a version of theorem by Kaimanovich proved in the case of trees \cite[Theorem 1.5.3]{K98}.

\begin{proposition}\label{pro-non-proper}
Assume the same setting as in Theorem \ref{acyl}.
If almost every trajectory $\{w_n o\}_{n=0}^\infty$ visits $\Oc(\o)o$ infinitely many times $\{\t_n\}_{i=0}^\infty$ such that $\t_{n+1}/\t_n \to 1$ as $n \to \infty$, then
the harmonic measure $\n_\m$ is exact dimensional and 
$$
\dim \n_\m=\frac{h_\m}{l_\m}.
$$
\end{proposition}

\begin{proof}
By the geodesic tracking property Theorem \ref{MTgeodesic} and the discussion above (2),
for almost every sample $\o$, there exists a $(1, C_\d)$-quasi-geodesic ray $\x_\o$ such that
$(1/n)d(w_n(\o) o, \x_\o) \to 0$ as $n \to \infty$.
We may assume that the $(1, C_\d)$-quasi-geodesic ray $\x_\o$ starts from $o$ by the quasi-geodesic stability (6.9.\ in \cite{V}).
Let us define the following event:
$$
A_{\e, N} 
:=\left\{ \o \in \O \ : \ 
\begin{aligned}
&\text{$d(w_n(\o) o, \x_\o(ln)) \le \e n$}, \\
&\text{$\m^{\ast n}\left(w_n(\o)\right) \le e^{-n(h-\e)}$ and $\t_{n+1} \le (1+\e)\t_n$ for all $n \ge N$}
\end{aligned}
\right\}.
$$
By Theorem \ref{MTgeodesic}, Theorem \ref{Shannon} and the assumption,
for every $\e>0$,
there exists $N_\e$ such that for all $N \ge N_\e$ we have $\Pr(A_{\e, N}) \ge 1-\e$.
Let $A_\e:=A_{\e, N_\e}$.
For $\n_\m$-almost every point $\x$ in $\partial X$, we denote a $(1, C_\d)$-quasi-geodesic ray starting from $o$ towards $\x$ by the same symbol $\x$.
Let $x_n:=\x(ln)$.
Given such a sequence $\{x_n\}_{n=0}^\infty$,
we will estimate the probability $\Pr(A_\e \cap\{w_n o \in B(x_n, 2\e n+C_0)\})$, where $C_0:=C_{R, \d}$.
For $\o \in A_\e$, for all $n \ge N_\e$, there is a time $\t_i \in [n, (1+\e)n]$ when $w_{\t_i}(\o)o$ visits $\Oc(\o)o$.
Since $d(o, w_n(\o) o) \le (l+\e)n+C_\d$ for $\o \in A_\e$, by Proposition \ref{MTO} which claims that
$$
\sharp\{g \in \G \ : \ go \in \Oc(\o)o \cap B(o, r)\} \le Cr,
$$
the number of possibilities of such $w_{\t_i}$ is at most $C'(l+\e)n$.
Therefore for all $n \ge N_0$, we have
$$
\Pr(A_\e \cap\{w_n o \in B(x_n, 2\e n+C_0)\}) \le C'(l+\e)ne^{-n(h-\e)}.
$$
Define the set $F_\e$ in $\partial X$ as in (\ref{defF}).
By Lemma \ref{hit} and the following Remark \ref{remarkhit}, as in the proof of Theorem \ref{F},
we have
$$
\e \n\left(F_\e \cap S(x_n,R)\right)\le \Pr\left(A_\e \cap \{w_n o \in B(x_n, 2\e n+C_0)\}\right),
$$
and thus
\begin{align*}
\liminf_{n \to \infty}\frac{\log \n\left(F_\e \cap S(x_n, R)\right)}{-n} \ge h-\e.
\end{align*}
By Lemma \ref{shadow}, we obtain
\begin{align*}
\liminf_{r \to 0}\frac{\log \n\left(F_\e \cap B(\x, r)\right)}{\log r} \ge \frac{h-\e}{l+\e}.
\end{align*}
This holds for $\n_\m$-almost every $\x$ in $\partial X$.
The rest follows as in the proof of Theorem \ref{exact}.
\end{proof}

\begin{proof}[Proof of Theorem \ref{acyl}]
We note that the probability that the identity element $id$ is in $\Oc(\o)$ is positive;
$\overline\Pr(\{\o \in \overline\O \ : \ id \in \Oc(\o)\})=p>0$ by \cite[The proof of Proposition 6.4]{MT}.
The event $\{\o \in \overline\O \ : \ w_n \in \Oc(\x_+, \x_-)\}$ coincides with $\{\o \in \overline\O \ : \ id \in \Oc(w_n^{-1}\x_+, w_n^{-1}\x_-)\}$ since $\Oc(\x_+, \x_-)$ is $\G$-equivariant, and also with $\{\o \in \overline\O \ : \ id \in \Oc(T^n\o)\}$ by definition of the shift $T$.
Since $(\overline \O, \overline \Pr, T)$ is a probability measure preserving system, we have
$$
\overline\Pr(\{\o \in \overline\O \ : \ id \in \Oc(T^n\o)\})=p>0.
$$
Hence by ergodicity of the system, almost every sample $\{w_n\}_{n=0}^\infty$ visits $\Oc(\o)$ infinitely often, and the $k$-th visiting time $\t_k$ satisfies that $k/\t_{k} \to p$ almost surely.
Therefore $\t_{k+1}/\t_k \to 1$ as $k \to \infty$, and by Proposition \ref{pro-non-proper}, the theorem follows.
\end{proof}

\section{Questions}\label{questions}

We collect some geometric measure theoretic questions about the harmonic measure on the boundary of a hyperbolic space.
We state questions for word hyperbolic groups although it can be stated in a more general setting.
Let $\G$ be a non-elementary hyperbolic group and $\partial \G$ be the boundary of a Cayley graph.

\begin{question}\label{q1}
Let $\m$ be a probability measure on $\G$ of finite first moment, and the semigroup $\sgr(\m)$ generated by the support of $\m$ be non-elementary.
Then is it true that 
the harmonic measure $\n_\m$ has the following property?: We have
$$
\dim \n_\m=\dim \partial \G
$$
if and only if $\n_\m$ and the Hausdorff measure of the right dimension $\dim \partial \G$ are mutually absolutely continuous.
\end{question}

This is true for $\m$ of finite support, more generally, of superexponential moment according to a result by Gou\"ezel (\cite{GMM}; see also \cite{BHM11} and \cite{T}).
We can also ask this question for the harmonic measure for Brownian motion on a regular covering of finite volume manifold as in Theorem \ref{Bm}.

An answer to the following question would be a step towards the above one:

\begin{question}
In the same setting as in Question \ref{q1}, when is the harmonic measure $\n_\m$ {\it doubling}?
Recall that we call the measure $\n$ on a metric space $Z$ {\it doubling} if there exists a constant $C>0$ such that for every $r \ge 0$ and every $z$ in $Z$ we have
$\n \left(B(z, 2r)\right) \le C \n \left(B(z, r)\right)$.
\end{question}

We still do not know if the harmonic measure $\n_\m$ is doubling for every $\m$ of finite exponential moment.

In Theorem \ref{exact}, \ref{exact-semigroup} and \ref{acyl}, the measure space $(\partial X, \n_\m)$ is actually the Poisson boundary for $(\G, \m)$, but the fact is not used to show the dimension formula: $\dim \n_\m=h_\m/l_\m$.
If the space $(\partial X, \n_\m)$ is not the Poisson boundary, then in the dimension formula, the entropy $h_\m$ would be replaced by the {\it differential $\m$-entropy} (see \cite{F} and discussion in the introduction in \cite{KLP}).

\begin{question}
Let $X$ be a hyperbolic proper geodesic metric space.
If $\G$ does not have exponential growth relative to $X$ (say, a countable dense subgroup of isometries), then the harmonic measure associated to a probability measure $\m$ on $\G$ is exact dimensional? If it is so, then what is the Hausdorff dimension?
\end{question}

Recently, Hochman and Solomyak have announced the dimension formula for a finitely generated dense subgroup of $SL(2, \R)$ \cite{Hochman-Solomyak}.

\textbf{Acknowledgements.}
The author would like to thank Pierre Mathieu for a discussion from which this work arises, owes subsequent discussions to him, Vadim A.\ Kaimanovich for helpful (and also stimulus) discussions and useful suggestions, 
Fran\c{c}ois Ledrappier for comments on historical background,
Takefumi Kondo for useful discussions,
S\'ebastien Alvarez, Behrang Forghani, and Yuval Peres for helpful feedback, and anonymous referees for reading the manuscript carefully and for helpful comments.
The author is supported by JSPS Grant-in-Aid for Young Scientists (B) 26800029.

\bibliographystyle{alpha}
\bibliography{dim}

\end{document}